\documentclass[12pt]{amsart}
\usepackage{amssymb,amsfonts,latexsym,amscd}

\usepackage[all,cmtip,2cell]{xy}
\UseAllTwocells
\usepackage{verbatim}

\newtheorem{theorem}{Theorem}[section]
\newtheorem{lemma}[theorem]{Lemma}
\newtheorem{proposition}[theorem]{Proposition}
\newtheorem{corollary}[theorem]{Corollary}
\theoremstyle{definition}

\newtheorem{claim}[theorem]{Claim}

\newtheorem{example}[theorem]{Example}
\newtheorem{question}[theorem]{Question}
\newtheorem{conjecture}[theorem]{Conjecture}
\newtheorem{remark}[theorem]{Remark}

\newcommand{\Sym}{\text{Sym}}
\newcommand{\SL}{\text{SL}}

\newcommand{\PGL}{\text{PGL}}
\newcommand{\Hom}{\text{Hom}}
\newcommand{\Ad}{\text{Ad}}

\newcommand{\gr}{\text{gr}}

\newcommand{\Rep}{\text{Rep}}

\newcommand{\Spec}{\text{Spec}}

\newcommand{\Corep}{\text{Corep}}
\newcommand{\Vect}{\text{Vect}}
\newcommand{\Der}{\text{Der}}

\newcommand{\g}{\mathfrak{g}}
\newcommand{\uu}{\mathfrak{u}}

\newcommand{\h}{\mathfrak{h}}

\newcommand{\ot}{\otimes}

\newcommand{\ben}{\begin{enumerate}}
\newcommand{\een}{\end{enumerate}}
\newcommand{\ad}{{\text{ad\,}}}

\newcommand{\Lie}{{\text{Lie}}}

\begin{document}

\title[Twisting of affine algebraic groups, I] {Twisting of affine algebraic groups, I}


\author{Shlomo Gelaki}
\address{Department of Mathematics, Technion-Israel Institute of
Technology, Haifa 32000, Israel} \email{gelaki@math.technion.ac.il}

\date{\today}


\begin{abstract}
We continue the study of twisting of affine algebraic groups $G$ (i.e., of Hopf $2-$cocycles $J$ for the function algebra $\mathcal{O}(G)$), which was started in \cite{EG1,EG2}, and initiate the study of the associated one-sided twisted function algebras $\mathcal{O}(G)_J$. We first show that $J$ is supported on a closed subgroup $H$ of $G$ (defined up to conjugation), and that $\mathcal{O}(G)_J$ is finitely generated with center $\mathcal{O}(G/H)$. We then use it to study the structure of $\mathcal{O}(G)_J$ for connected nilpotent $G$. We show that in this case $\mathcal{O}(G)_J$ is a Noetherian domain, which is a simple algebra if and only if $J$ is supported on $G$, and describe the simple algebras that arise in this way. We also use \cite{EG2} to obtain a classification of Hopf $2-$cocycles for connected nilpotent $G$, hence of fiber functors $\Rep(G)\to \Vect$. Along the way we provide many examples, and at the end formulate several ring-theoretical questions about the structure of the algebras $\mathcal{O}(G)_J$ for arbitrary $G$.
\end{abstract}

\maketitle

\section{Introduction}
The theory of twisting of finite groups was developed by Movshev in \cite{M}. The main achievement of this theory is the classification of twists (equivalently, {\em Hopf $2-$cocycles}) for finite groups. More precisely, Movshev proved that twists $J$ for a finite group $G$ correspond to pairs $(H,\omega)$, where $H$ is a central type subgroup of $G$ (= the {\em support} of $J$) and $\omega\in H^2(H,\mathbb{C}^*)$ is a nondegenerate $2-$cocycle (see also \cite{EG3}). In the language of tensor categories, this classification is equivalent to the classification of {\em fiber functors} $\Rep(G)\to \Vect$ of the {\em tensor category} $\Rep(G)$ of finite dimensional representations of $G$ \cite{EG3}. 

It is a natural problem to try to extend Movshev's theory to {\em affine algebraic groups} $G$ over $\mathbb{C}$, and classify Hopf $2-$cocycles for the function Hopf algebra $\mathcal{O}(G)$ of $G$. This was manifested in \cite{EG1} where, among other things, it is explained how to construct some Hopf $2-$cocycles for $\mathcal{O}(G)$ from solutions to the CYBE. Later on it was proved in \cite[Corollary 3.5]{EG2} that for unipotent $G$, Hopf $2-$cocycles $J$ for $\mathcal{O}(G)$ are classified by pairs $(H,\omega)$, where $H$ is a closed subgroup of $G$ (= the {\em support} of $J$) and $\omega\in \wedge^2\Lie(H)^*$ is a nondegenerate $2-$cocycle. However, a general classification of Hopf $2-$cocycles for affine algebraic groups is still unknown. 

Like in the finite case, a closely related problem is to classify fiber functors $\Rep(G)\to \Vect$ of the tensor category $\Rep(G)$ of finite dimensional rational representations of $G$. In the infinite case this problem generalizes (but is {\em not} equivalent to) the first one, since Hopf $2-$cocycles for $\mathcal{O}(G)$ correspond to {\em classical} fiber functors of $\Rep(G)$ (i.e., tensor functors which preserve dimensions). 
For $G=\SL_2$ it was done by Bichon \cite{B1,B2}, and for classical fiber functors for $G=\SL_3$ by Ohn \cite{O1,O2}.

Another aspect of Movshev's theory for finite groups involves the
study of the dual algebra $(\mathbb{C}[G]_J)^*$ to the one-sided twisted group coalgebra $\mathbb{C}[G]_J$. This algebra is semisimple, and it is simple (i.e., a matrix algebra) precisely when $J$ is supported on $G$ \cite{M}.
In the algebraic group case the (infinite dimensional) algebra $\mathcal{O}(G)_J$ (where the multiplication in $\mathcal{O}(G)$ is deformed from the right by $J$) is much more interesting and complicated, so it is natural to study its structure, find out whether it is a simple algebra if and only if $J$ is  supported on $G$ (as in the finite case), and determine the structure of the simple algebras which can be
realized in this way. One could hope that answering these questions would help to classify Hopf $2-$cocycles for $\mathcal{O}(G)$, as it did in the finite case. One of the main goals of this paper is to do so for connected nilpotent algebraic groups. Other types of affine algebraic groups (e.g., solvable, reductive) will be considered in future publications. Meanwhile, we also obtain some initial results for general affine algebraic groups.

The organisation of the paper is as follows. In Section 2 we recall some basic notions and results used in the sequel. 

In Section 3 we consider Hopf $2-$cocycles $J$ for arbitrary affine algebraic groups $G$ over $\mathbb{C}$. We first show in Theorem \ref{support} that $J$ is supported on a closed subgroup $H$ of $G$ (defined up to conjugation), i.e., that $J$ is gauge equivalent to a {\em minimal} Hopf $2-$cocycle for $\mathcal{O}(H)$. We then show in Proposition \ref{fg} that $\mathcal{O}(G)_J$ is finitely generated, and in Theorem \ref{trivcen} that its center is isomorphic to $\mathcal{O}(G/H)$. Consequently, if $J$ is minimal then the algebra $\mathcal{O}(G)_J$ has trivial center, and the converse holds provided that $G/H$ is an affine variety. 

In Section 4 we use Theorem \ref{trivcen} to study the structure of the algebra $\mathcal{O}(G)_J$ for connected nilpotent $G$. We start with tori in Theorem \ref{noethdomtori}, proceed with unipotent groups in Theorem \ref{noethdomunip}, and finally treat the general casen in Theorem \ref{noethdomnilp}. We show that $\mathcal{O}(G)_J$ is a Noetherian domain, and that the simple algebras which arise in this case are certain twisted group algebras of finitely generated abelian groups (for tori), Weyl algebras (for unipotent groups), and certain crossed products of them (for general $G$). In particular, we see that in the nilpotent case, $J$ is minimal if and only if the algebra $\mathcal{O}(G)_J$ has trivial center, if and only if, the algebra $\mathcal{O}(G)_J$ is simple. 

In Section 5 we use Theorem \ref{noethdomunip} to give in Theorem \ref{qfunip} and Corollary \ref{revis} an alternative proof (which does not use Etingof--Kazhdan quantization theory \cite{EK1,EK2,EK3}) to the classification of Hopf $2-$cocycles for unipotent algebraic groups given in \cite[Theorem 3.2]{EG2}. Then in Theorem \ref{nilp} and Proposition \ref{solcybe} we use the classification of unipotent fiber functors for affine algebraic groups given in \cite[Theorem 3.2]{EG2} to obtain a classification of Hopf $2-$cocycles for connected nilpotent algebraic groups in terms of solutions to the CYBE. 

In Section 6 we conclude the paper with some natural ring-theoretical questions on the structure of the algebras $\mathcal{O}(G)_J$ for arbitrary affine algebraic groups over $\mathbb{C}$, and answer them in the case where the support of $J$ is a finite subgroup of $G$.

{\bf Acknowledgments.} I am grateful to Pavel Etingof for stimulating and helpful discussions. Part of this work was done while I was on Sabbatical in the Department of Mathematics at the University of Oregon in Eugene; I am grateful for their warm hospitality. This research was partially supported by The Israel Science Foundation (grant No. 561/12).

\section{Preliminaries} 

Throughout the paper we shall work over the field $\mathbb{C}$ of complex numbers.

\subsection{Hopf $2-$cocycles}\label{hopftwococycle} 
Let $H$ be a Hopf algebra over $\mathbb{C}$. A linear map 
$J:H\ot H\to \mathbb{C}$
is called a {\em Hopf $2-$cocycle} for $H$ if it has an inverse
$J^{-1}$ under the convolution product $*$ in $\Hom _{\mathbb{C}} (H\ot
H,\mathbb{C})$, and satisfies the two conditions
\begin{align}\label{2coc}
\sum J (a_1b_1,c)J (a_2,b_2)&=\sum J
(a,b_1c_1)J (b_2,c_2),\\ J (a,1) =\varepsilon(a)&=J (1,a)\nonumber
\end{align}
for all $a,b,c\in H$ (see, e.g., \cite{d,ma}). 

Given a Hopf $2-$cocycle $J$ for $H$, one can construct a new Hopf
algebra $H^{J}$ as
follows. As a coalgebra $H^{J}=H$, and the new multiplication $m^{J}$ is given by
\begin{equation}\label{nm}
m^{J}(a\ot b):=\sum J^{-1} (a_1,b_1)a_2b_2J
(a_3,b_3)
\end{equation}
for all $a,b\in H$.
Equivalently, $J$ defines a tensor structure on the forgetful functor $\Corep(H)\to \Vect$, where $\Corep(H)$ is the tensor category of comodules over $H$.

\subsection{Cotriangular Hopf algebras}\label{cothopalg} Recall that $(H,R)$ is a {\em cotriangular} Hopf algebra if $R:H\ot H\to \mathbb{C}$ is an invertible linear map under $*$, such that 
\begin{itemize}
\item
$\sum R(h_1,g_1)R(g_2,h_2)=\varepsilon(g)\varepsilon(h)$ (i.e., $R^{-1}=R_{21}$),
\item 
$R(h,gl)=\sum R(h_1,g)R(h_2,l)$,
\item
$R(hg,l)=
\sum R(g,l_1)R(h,l_2)$, and
\item
$\sum R(h_1,g_1)g_2h_2=\sum h_1g_1R(h_2,g_2)$
\end{itemize}
for every $h,g,l\in H$. Given a Hopf $2-$cocycle $J$ for $H$, $(H^{J},R^J)$ is also cotriangular, where $R^{J}:=J_{21}^{-1}*R*J$.

\begin{proposition}\label{minimal}
Let $(H,R)$ be a cotriangular Hopf algebra, and let $I:=\{a\in H\mid R(b,a)=0,\,b\in H\}$ be the right radical of $R$. Then $I$ coincides with the left radical $\{b\in H\mid R(b,a)=0,\,a\in H\}$ of $R$, and is a Hopf ideal of $H$.
\end{proposition}

\begin{proof}
Since $R(a,b)=R^{-1}(b,a)=R(S(b),a)$, we see that $I$ coincides with the left radical of $R$, and the claims follow.
\end{proof}

A cotriangular Hopf algebra $(H,R)$ such that $R$ is nondegenerate (i.e., $I=0$) is called {\em minimal}. By Proposition \ref{minimal}, any cotriangular Hopf algebra $(H,R)$ has a unique minimal cotriangular Hopf algebra quotient, which we shall denote by $(H_{min},R)$. 

Note that in the minimal case, the nondegenerate form $R$ defines two injective Hopf algebra maps $H\xrightarrow{1-1} H^*_{\rm fin}$ from $H$ into its finite dual Hopf algebra $H^*_{\rm fin}$.

\subsection{Ore extensions}\label{oreext} Let $A$ be a unital associative ring, and let $\delta$ be a derivation of $A$ (i.e., $\delta:A\to A$ is a homomorphism of abelian groups such that $\delta(ab)=\delta(a)b+a\delta(b)$ for every $a,b\in A$). Recall that the {\em Ore extension} of $A$ with respect to $\delta$ is defined to be the ring $A[z;\delta]$ generated over $A$ by $z$ subject to the relations $za-az=\delta(a)$ for every $a\in A$.

\subsection{Quasi-Frobenius Lie algebras}\label{drinres} Recall that a Lie algebra $\h$ is called {\em quasi-Frobenius with symplectic form $\omega$} if $\omega\in\wedge^2\h^*$ is a {\em nondegenerate} 2-cocycle, i.e., $\omega:\h\times \h\to \mathbb{C}$ is a nondegenerate skew-symmetric bilinear form satisfying
\begin{equation}\label{2cocy}
\omega([x,y],z)+\omega([z,x],y)+\omega([y,z],x)=0
\end{equation}
for every $x,y,z\in \h$.

Recall that an element $r\in \wedge^2 \g$ is a solution to the classical Yang-Baxter equation (CYBE) if
\begin{equation}\label{cybe}
[r_{12},r_{13}]+[r_{12},r_{23}]+[r_{13},r_{23}]=0.
\end{equation}
By Drinfeld \cite{dr}, solutions $r$ to the CYBE (\ref{cybe}) in $\wedge^2\g$
are classified by pairs $(\h,\omega)$, via $r=\omega^{-1}\in \wedge^2\h$, where $\h\subseteq \g$ is a quasi-Frobenius Lie
subalgebra with symplectic form $\omega$. 

\subsection{Central extensions of Lie algebras}\label{rext}
A $1-$dimensional central extension $\hat{\mathfrak{g}}$ of a Lie algebra $\mathfrak{g}$ is a short exact sequence of Lie algebras $$0\to \mathbb{C}\to \hat{\mathfrak{g}}\to \mathfrak{g}\to 0.$$ Recall that equivalence classes of such extensions are in one to one correspondence with the elements of the group $H^2(\mathfrak{g},\mathbb{C})$. Given a $2-$cocycle $\omega\in \wedge^2 \mathfrak{g}^*$, the bracket on $\hat{\mathfrak{g}}$ is given by $$[(x,\alpha ),(y,\beta )]=([x,y],\omega(x,y)),$$
for every $x,y\in \g$ and $\alpha,\beta\in\mathbb{C}$.

\section{Movshev's theory for affine algebraic groups} 

Let $G$ denote an affine algebraic group over $\mathbb{C}$, and let $J$ be a Hopf $2-$cocycle for $\mathcal{O}(G)$ (see Section \ref{hopftwococycle}). In this section we introduce the support of $J$ in $G$, and describe the center of the infinite dimensional algebra $\mathcal{O}(G)_J$ in terms of it.

\subsection{The support of $J$} Let us say that $J$ is {\em minimal} if the cotriangular Hopf algebra $(\mathcal{O}(G)^J,R^{J})$ is minimal (see Section 2.2).

\begin{theorem}\label{support} 
There exist a closed subgroup $H$ of $G$ and a minimal Hopf $2-$cocycle $J'$ for $\mathcal{O}(H)$ such that $J$ is gauge equivalent to $J'$.
\end{theorem}

\begin{proof}
Let $(\mathcal{O}(G)^{J}_{min},R^{J})$ be the minimal cotriangular Hopf algebra quotient of $(\mathcal{O}(G)^{J},R^J)$ (see Section 2.2). Clearly, the restriction functor $$\Corep(\mathcal{O}(G)^{J})\to \Corep(\mathcal{O}(G)^{J}_{min})$$ is a surjective symmetric tensor functor. Reasoning as in the proof of \cite[Proposition 6.5]{G} (see also \cite[Proposition 1]{Be}), we get that $\Corep(\mathcal{O}(G)^{J}_{min})$ is equivalent to $\Corep(\mathcal{O}(H))$, as a symmetric tensor category, for some closed subgroup $H$ of $G$ (defined up to conjugation). Reasoning now as in the proof of \cite[Proposition 6.7]{G} we see that such an equivalence functor gives rise to a Hopf $2-$cocycle $J'$ for $\mathcal{O}(H)$, and an isomorphism of cotriangular Hopf algebras $$(\mathcal{O}(G)^{J'},R^{J'})\xrightarrow{\cong}(\mathcal{O}(G)^{J},R^J)$$ (viewing $J'$ as a Hopf $2-$cocycle for $\mathcal{O}(G)$), which implies that $J$ is gauge equivalent to $J'$.
\end{proof}

The closed subgroup $H$ of $G$ (defined up to conjugation) will be called the {\em support} of the Hopf $2-$cocycle $J$. We shall say that $J$ has {\em finite rank} if its support $H$ is a finite subgroup of $G$ (equivalently, $J$ becomes a form of finite rank after a gauge transformation). 

\begin{remark}
In the language of \cite{G}, Hopf $2-$cocycles for $G$ of finite rank correspond to {\em geometrical} fiber functors $\Rep(G)\to \Vect$ of $\Rep(G)$.
\end{remark} 

\subsection{The algebra $\mathcal{O}(G)_J$} Let $\mathcal{O}(G)_J$ be the twisted-on-the-right function algebra of $G$, and let $_{J^{-1}} \mathcal{O}(G)$ be the twisted-on-the-left function algebra of $G$. Both $\mathcal{O}(G)_J$ and $_{J^{-1}} \mathcal{O}(G)$ are {\em torsor $G-$algebras}, i.e., they are unital associative algebras with multiplication rule
\begin{equation}\label{multj}
m_J(a\ot b)=\sum a_1b_1J(a_2,b_2)
\end{equation}
and
\begin{equation}\label{multjinv}
_{J^{-1}} m(a\ot b)=\sum J^{-1}(a_1,b_1)a_2b_2
\end{equation} 
respectively, equipped with a $G-$action such that $\mathcal{O}(G)_J$ and $_{J^{-1}} \mathcal{O}(G)$ are isomorphic to the regular representation of $G$ as $G-$modules. 

\begin{remark}
If $G$ is a {\em finite} group then the algebras $\mathcal{O}(G)_J$ and $_{J^{-1}} \mathcal{O}(G)$ are semisimple, and $(\mathcal{O}(G)_J)^{op}\cong _{J^{-1}}\mathcal{O}(G)$ \cite{M}. It is known that $\mathcal{O}(G)_J$ is simple if and only if $J$ is minimal, and that in this case $G$ is a group of central type (\cite{M}, \cite{EG3}).
It is thus natural to study ring-theoretical properties of the (infinite
dimensional) algebra $\mathcal{O}(G)_J$, e.g., to find out when it is a Noetherian domain, and when it is simple (and what is the structure of the simple algebras which can be realized in this way).
\end{remark}

\begin{proposition}\label{fg}
The algebra $\mathcal{O}(G)_J$ is finitely generated.
\end{proposition}

\begin{proof}
It is straightforward to verify that if $f_1,\dots,f_n$ are generators of the algebra $\mathcal{O}(G)$ then 
$\mathcal{O}(G)_J$ is generated by the components of their comultiplications $\Delta(f_1),\dots,\Delta(f_n)$. Indeed, it suffices to verify that every monomial in $f_1,\dots,f_n$, in the algebra $\mathcal{O}(G)$, is a linear combination of monomials in the components of $\Delta(f_1),\dots,\Delta(f_n)$, in the algebra $\mathcal{O}(G)_J$. But this follows by simple induction using the fact that by (\ref{multj}), $fg=\sum f_1\cdot g_1J^{-1}(f_2,g_2)$ for every $f,g\in \mathcal{O}(G)$, where $\cdot$ denotes the product in $\mathcal{O}(G)_J$.
\end{proof}


\subsection{The center of $\mathcal{O}(G)_J$} 
Let $Z$ be the center of $\mathcal{O}(G)_J$; it is a commutative algebra with a $G-$action, so $\Spec(Z)$ is an affine scheme with a $G-$action. The next theorem shows that $\Spec(Z)$ is a $G-$homogeneous space with the support of $J$ being the stabilizer of a closed point.

\begin{theorem}\label{trivcen}
The center $Z$ of the algebra $\mathcal{O}(G)_J$ is isomorphic to $\mathcal{O}(G/H)$, where $H$ is the support of $J$. In particular, if $J$ is minimal then  the center of $\mathcal{O}(G)_{J}$ is trivial, and if $G/H$ is affine then the converse holds too.
\end{theorem}

\begin{proof}
Let us first show that the algebra $\mathcal{O}(H)_{J}$ has trivial center. Indeed, if $\varphi$ is central in $\mathcal{O}(H)_{J}$ then for every $\psi\in \mathcal{O}(H)_{J}$, $$0=\varphi \cdot\psi-\psi\cdot\varphi=\sum \varphi_1\psi_1(J-J_{21})(\varphi_2,\psi_2).$$ Evaluating at the unit element $e\in H$, we get $J(\varphi,\psi)=J_{21}(\varphi,\psi)$, so $J=J_{21}$ on $\mathcal{O}(H)\ot Z$.
 
Since $Z$ is $H-$stable, $h\varphi=\sum \varphi_1(h)\varphi_2\in Z$ for every $h\in H$, and hence $Z$ is a left coideal in $\mathcal{O}(H)$. Consider the left module $\Hom _{\mathbb{C}} (\mathcal{O}(H)\ot Z,\mathbb{C})$ over the algebra $\Hom _{\mathbb{C}} (\mathcal{O}(H)\ot \mathcal{O}(H),\mathbb{C})$ (with respect to the convolution products $*$).
We have $J^{-1}*J_{21}=J^{-1}*J=\varepsilon\ot \varepsilon$ in $\Hom _{\mathbb{C}} (\mathcal{O}(H)\ot Z,\mathbb{C})$ (as $J=J_{21}$ on $\mathcal{O}(H)\ot Z$), where we view 
$J^{-1}$ as an element of $\Hom _{\mathbb{C}} (\mathcal{O}(H)\ot \mathcal{O}(H),\mathbb{C})$, and $J_{21},J$ as elements of $\Hom _{\mathbb{C}} (\mathcal{O}(H)\ot Z,\mathbb{C})$. Thus by the nondegeneracy of $R^J$, $\varphi =\varphi(e)1$ is a constant function, as claimed.

Now, since the restriction map $\mathcal{O}(G)_{J}\twoheadrightarrow \mathcal{O}(H)_{J}$ is an algebra epimorphism, it maps the center $Z$ of $\mathcal{O}(G)_{J}$ onto $\mathbb{C}$. Take $\varphi\in Z$, and $g\in G$. Since $Z$ is $G-$stable, $g^{-1}\varphi\in Z$ too, hence is constant on $H$. But $(g^{-1}\varphi)(h)=\varphi(gh)$ for every $h\in H$, and hence $\varphi$ is constant on every left coset $gH$. We thus get an algebra injection $Z\hookrightarrow \mathcal{O}(G/H)$.
In particular, we may identify $Z$ with a subalgebra of $\mathcal{O}(G/H)\subseteq \mathcal{O}(G)$.

Conversely, since $\mathcal{O}(G/H)$ is a left coideal subalgebra of $\mathcal{O}(G)$, it follows from (\ref{multj}) that $\mathcal{O}(G/H)$ is contained in $Z$. 

We have therefore proved that $Z=\mathcal{O}(G/H)$, as required.
\end{proof}

\begin{remark}\label{Q}
Set $Q:=J-J_{21}$. Note that in the proof of Theorem \ref{trivcen} we obtained that if either $\varphi$ or $\psi$ are central in $\mathcal{O}(G)_{J}$ then $Q(\varphi,\psi)=0$. If in addition either $\varphi(e)=0$ or $\psi(e)=0$ then $R^J(\varphi,\psi)=0$.
\end{remark}

\begin{example}
Here is an example where $J$ is not minimal but the center $\mathcal{O}(G/H)$ of $\mathcal{O}(G)_J$ is trivial (so the quotient variety $G/H$ is not affine). 

Let $G=\SL_2$, and $H=\mathbb{G}_m\ltimes \mathbb{G}_a$ be a Borel subgroup of $G$ (the group of affine transformations of the line). The Lie algebra $\h$ of $H$ is spanned by two elements $X:=x\frac{\partial}{\partial x},Y:=x\frac{\partial}{\partial y}$ such that $[X,Y]=Y$. It is straightforward to check that
$$
J(h):=\sum_{n\ge 0} \frac{h^n}
{n!} X(X-1)\cdots(X-n+1)\ot Y^n
$$
is a minimal Hopf $2-$cocycle for $\mathcal{O}(H)$ for every $h\in\mathbb{C}^*$ (see \cite[Example 5.2]{EG1}), and hence can be viewed as a non-minimal Hopf $2-$cocycle for $\mathcal{O}(G)$.
%
\end{example}

\section{Movshev's theory for affine nilpotent algebraic groups}\label{unip} 

In this section we study the structure of the infinite dimensional algebras $\mathcal{O}(G)_J$ for affine connected nilpotent algebraic groups $G$ over $\mathbb{C}$.

\subsection{Tori} Let $T$ be an algebraic torus over $\mathbb{C}$ of dimension $k$, and let $X(T)$ be its character group. Recall that $X(T)$ is a finitely generated torsion-free abelian group of rank $k$, and that $\mathcal{O}(T)$ is isomorphic to the group Hopf  algebra $\mathbb{C}[X(T)]=\mathbb{C}[x_1^{\pm 1},\dots,x_k^{\pm 1}]$ (= Laurent polynomial Hopf algebra).

\begin{theorem}\label{h2ctor}
Let $T$ be an algebraic torus over $\mathbb{C}$. Gauge equivalence classes of Hopf $2-$cocycles for $\mathcal{O}(T)$ are in bijection with elements of the group $H^2(X(T),\mathbb{C}^*)=\Hom(\wedge^2 X(T),\mathbb{C}^*)$ of alternating bicharacters on $X(T)$. Furthermore, minimal Hopf $2-$cocycles correspond under this bijection to nondegenerate alternating bicharacters .
\end{theorem}

\begin{proof}
Follows from the definitions in a straightforward manner.
\end{proof}

\begin{remark}
The dimension of an algebraic torus having a minimal Hopf $2-$cocycle need not be even, as the following example shows. Let $T$ be a $3-$dimensional torus, with $X(T)$ having a basis $x_1,x_2,x_3$. Consider the alternating bicharacter on $X(T)$ given by $B(x_1,x_2)=a$, $B(x_3,x_1)=b$, and $B(x_2,x_3)=c$. If $a,b,c$ are generic (e.g., multiplicatively independent), the kernel of this form on $\mathbb{Z}^3$ is trivial, so the corresponding twist is minimal. 
(Morally $J$ corresponds to a $2-$dimensional Lie subgroup of $T$, which is not closed -- it is a ``winding" of the torus.)
\end{remark}

Given an $k\times k$ matrix $q=\{q_{ij}\}$ such that $q_{ii}=1$ and $q_{ji}=q_{ij}^{-1}$, let $E(q)$ denote the algebra generated by $x_1^{\pm 1},\dots,x_k^{\pm 1}$ subject to the relations $x_ix_j=q_{ij}x_jx_i$. It is well known that $E(q)$ is a simple algebra if and only if its center is $\mathbb{C}$ (see e.g., \cite[Proposition 1.3]{MP}).

\begin{theorem}\label{noethdomtori}
Let $T$ be an algebraic torus over $\mathbb{C}$ of dimension $k$. Let $J$ be a Hopf $2-$cocycle for $\mathcal{O}(T)$, let $S$ be its support, and let $l$ be the dimension of $S$. Fix a splitting of the inclusion $S^0\hookrightarrow T$, and use it to embed $\mathcal{O}(S^0)$ in $\mathcal{O}(T)$. Set $\lambda_{ij}:=R^J(x_i,x_j)$, $1\le i,j\le k$, let $\lambda=\{\lambda_{ij}\}$, and let $\lambda^0$ be the $l\times l$ matrix corresponding to the restriction of $R^J$ to $\mathcal{O}(S^0)$. The following hold:

1) $\mathcal{O}(T)_J\cong E(\lambda)$ is a Noetherian domain with center $\mathcal{O}(T/S)$.

2) $S$ is the Zariski-closure of the image of the map $$\tilde{\lambda}:X(T)\to \Hom(X(T),\mathbb{C}^*)=T,\,\,\,x\mapsto \lambda(x,?)$$ (viewing $\lambda$ as an alternating bicharacter on $X(T)$).

3) $\mathcal{O}(S)_J\cong \mathbb{C}^{R^J}[X(S)]\cong E(\lambda^0) \#_J \mathbb{C}[X(S/S^0)]$ (a crossed product with $2-$cocycle $J$) is a simple algebra. 

4) $(\mathcal{O}(T)_J)^{op}\cong _{J^{-1}} \mathcal{O}(T)$. 
\end{theorem}

\begin{proof}
1) By (\ref{multj}), $x_i\cdot x_j=J(x_i,x_j)x_ix_j$, hence $\mathcal{O}(T)_J$ is isomorphic to the twisted group algebra $\mathbb{C}^{\lambda}[X(T)]$ generated by $x_1^{\pm 1},\dots,x_k^{\pm 1}$ subject to the relations $x_i\cdot x_j=\lambda_{ij} (x_j\cdot x_i)$, i.e., to $E(\lambda)$. It is well known that $E(\lambda)$ is a Noetherian domain (see e.g., \cite{MR}). Finally, the center of $\mathcal{O}(T)_J$ is $\mathcal{O}(T/S)$ by Theorem \ref{trivcen}.

2) Let $\Gamma:=\ker(\tilde{\lambda})=\{x\in X(T)\mid \lambda(x,?)\equiv 1\}$. Clearly, the center of $\mathbb{C}^{\lambda}[X(T)]$ is the group algebra $\mathbb{C}[\Gamma]$, so $\mathcal{O}(T/S)=\mathbb{C}[\Gamma]$. Therefore, $\Gamma=X(T/S)$ (as $\Gamma$ is the character group of a quotient of $T$).

3) Fix a splitting of the inclusion $S^0\hookrightarrow S$, and use it to embed $X(S^0)$ in $X(S)$. Then $X(S)=X(S^0)\times X(S/S^0)$ as groups, and it is straightforward to verify that $\mathcal{O}(S)_J\cong E(\lambda^0)\ot \mathbb{C}^J[X(S/S^0)]$ as vector spaces, $E(\lambda^0)$ and $\mathbb{C}^J[X(S/S^0)]$ are subalgebras of $\mathcal{O}(S)_J$, the group $X(S/S^0)$ acts on $\mathcal{O}(S^0)_J$ via automorphisms, and hence that $\mathcal{O}(S)_J$ is a crossed product as described.

We claim that $\lambda$ is nondegenerate on $X(S^0)$. Indeed, if $x\in X(S^0)$ is such that $\lambda(x,?)\equiv 1$ on $X(S^0)$ then $\lambda(x^n,?)\equiv 1$ on $X(S)$, where $n$ is the order of $S/S^0$. But this implies that $x^n=1$, and hence $x=1$. Thus by 1), $\mathcal{O}(S^0)_J\cong E(\lambda^0)$ is a simple algebra.

Finally, since $\mathcal{O}(S)_J=\mathbb{C}^{R^J}[X(S)]$ is a twisted group algebra with a nondegenerate $2-$cocycle $R^J$, it is simple.



4) Follows from (\ref{multj}) and (\ref{multjinv}).
\end{proof}

\begin{example}\label{multgr2}
Let $T=\mathbb{G}_m\times \mathbb{G}_m$. Then $A=\mathcal{O}(T)=\mathbb{C}[x^{\pm 1},y^{\pm 1}]$, where $x$ and $y$ are grouplike elements. The Lie algebra of $T$ is abelian of dimension $2$, with basis $\delta:=x\frac{\partial}{\partial x}$ and $\mu:=y\frac{\partial}{\partial y}$. For every $h\in \mathbb{C}^*$ let $r_h:=\pi ih(\delta \wedge \mu)$, and set $q:=e^{2\pi ih}$. Let $\mathcal{O}(T)\ot \mathcal{O}(T)\xrightarrow{\eta_q}\mathcal{O}(T)\ot \mathcal{O}(T)$ denote the action by $e^{r_h}$ (it is well defined). Then as it is explained in \cite[Section 4]{EG1}, the composition
$$J_q:\mathcal{O}(T)\ot \mathcal{O}(T)\xrightarrow{\eta_q}\mathcal{O}(T)\ot \mathcal{O}(T)\xrightarrow{\varepsilon\ot \varepsilon}\mathbb{C}$$  
is a Hopf $2-$cocycle for $\mathcal{O}(T)$. Explicitly,
$J_q:\mathcal{O}(T)\ot \mathcal{O}(T)\to\mathbb{C}$ is given by  
\begin{eqnarray*}
\lefteqn{J_q(f,g)=e^{r_h}(f\ot g)_{|(1,1)}}\\
& = & f(1)g(1) + \pi ih\left(f_x(1)g_y(1)-f_y(1)g_x(1)\right)+\cdots 
\end{eqnarray*}
for every $f,g\in \mathbb{C}[x^{\pm 1},y^{\pm 1}]$. Hence, $J_q(x,y)=q^{1/2}$, $J_q(y,x)=q^{-1/2}$, and $x\cdot y=q(y\cdot x)$. So, $\lambda_{12}=q$. 

There are two cases: 

1) If $h\in \mathbb{Q}$, with denominator $d$, then $q$ is a root of $1$ of degree $d$, and so $J_q$ is supported on $S=\mathbb{Z}/d\mathbb{Z}\times \mathbb{Z}/d\mathbb{Z}$. We have $\mathcal{O}(S)_{J_q}\cong Mat_d(\mathbb{C})$, and the center of $\mathcal{O}(T)_{J_q}$ is the subalgebra generated by $x^{\pm d},y^{\pm d}$ (which is isomorphic to $\mathcal{O}(T/S)$).

2) If $h\notin \mathbb{Q}$, then $S=T$, i.e., $J_q$ is a minimal Hopf $2-$cocycle for $\mathcal{O}(T)$, and $\mathcal{O}(T)_{J_q}\cong E(\lambda)$ is a simple algebra. 
\end{example}

\subsection{Unipotent groups} Let $U$ be a unipotent algebraic group over $\mathbb{C}$ of dimension $m$. 
Recall that $A:=\mathcal{O}(U)$ is a finitely generated commutative 
irreducible pointed Hopf algebra, which is isomorphic to a
polynomial algebra as an algebra. Recall also that since $U$ is obtained from $m$ successive $1-$dimensional central extensions with kernel $\mathbb{G}_a$ (= additive group), $A$ 
admits a filtration 
\begin{equation}\label{filt}
\mathbb{C}=A_0\subset A_1\subset\cdots\subset A_i\subset\cdots\subset A_{m}=A
\end{equation}
by Hopf subalgebras $A_i$ such that for every $1\le i\le m$, $A_i=\mathbb{C}[z_1,\dots,z_i]$ is a polynomial algebra and
\begin{equation}\label{com}
\Delta(z_i)=z_i\ot 1 + 1\ot z_i + Z_i,
\end{equation}
where $Z_i\in A_{i-1}^+\ot A_{i-1}^+$, with $Z_1=Z_2=0$. We shall sometimes write $Z_i=\sum Z_i'\ot Z_i''$.

\begin{lemma}\label{lemmaunip}
Let $U$ be a unipotent algebraic group over $\mathbb{C}$ of dimension $m$, and  
let $J$ be a Hopf $2-$cocycle for $A:=\mathcal{O}(U)$. Let $\cdot$ denote the multiplication in $A_J$, and set $Q:=J-J_{21}$. The following hold:

1) The filtration on $A$ given in (\ref{filt}) determines an algebra filtration on the algebra $A_J$:
\begin{equation}\label{filt1}
\mathbb{C}=A_0\subset A_1\subset\cdots \subset(A_i)_J\subset\cdots\subset (A_{m})_J=A_J.
\end{equation}

2) For every $i$, the algebra $(A_i)_J$ is generated by $z_i$ over $(A_{i-1})_J$.

3) For every $j<i$, 
\begin{eqnarray*}
\lefteqn{z_i\cdot z_j - z_j\cdot z_i =
Q(z_i,z_j)1}\\ & + & \sum Z_i'Q(Z_i'',z_j) + \sum Z_j'Q(z_i,Z_j'') + \sum Z_i'Z_j'Q(Z_i'',Z_j'')
\end{eqnarray*}
is an element of $A_{i-1}$.

4) For every $i$, $(A_{i})_J=(A_{i-1})_J[z_i;\delta_i]$ is an Ore extension (see Section \ref{oreext}).

\end{lemma}

\begin{proof}
1) Follows since each $A_i$ is a Hopf subalgebra of $A$.

2) Follows from Remark \ref{fgexp} and (\ref{com}).

3) Follows from (\ref{multj}) and (\ref{com}).

4) By Part 3), $$\delta_i:(A_{i-1})_J \to (A_{i-1})_J,\,\,r\mapsto z_i\cdot r - r\cdot z_i,$$ is a derivation of $(A_{i-1})_J$, so the claim follows from Part 2).
\end{proof}

\begin{proposition}\label{support1ext}
Let $V$ be a unipotent algebraic group over $\mathbb{C}$, and let $$1\to \mathbb{G}_a\xrightarrow{\iota} V\xrightarrow{\pi} \bar{V}\to 1$$ be a central extension. Let $J$ be a minimal Hopf $2-$cocycle for $\mathcal{O}(V)$, and let $\bar{J}:=J\circ(\pi^*\ot \pi^*)$ be the restriction of $J$ to $\mathcal{O}(\bar{V})$. There exists a closed subgroup $\iota(\mathbb{G}_a)\subseteq L\subset V$ of codimension $1$ (hence normal) such that $\bar{L}:=\pi(L)\subset \bar{V}$ is the support of $\bar{J}$.
In particular, $\bar{L}$ is normal in $\bar{V}$ and $\bar V/\bar{L}\cong \mathbb{G}_a$.
\end{proposition}

\begin{proof}
Consider the inclusion $\pi^*:(\mathcal{O}(\bar{V})^{\bar{J}},R^{\bar{J}})\hookrightarrow (\mathcal{O}(V)^J,R^J)$ of cotriangular Hopf algebras, and let $\mathcal{I}$ be the space of all $\alpha\in \mathcal{O}(V)$ such that $R^J(\alpha,\pi^*(\beta))=0$ for every $\beta\in \mathcal{O}(\bar{V})$ (the ``complement" of $\mathcal{O}(\bar{V})$ with respect to the nondegenerate bilinear form $R^J$ on $\mathcal{O}(V)$). Since $\mathcal{I}$ is a Hopf ideal of $\mathcal{O}(V)^J$, $\mathcal{I}^{J^{-1}}$ is a Hopf ideal of $\mathcal{O}(V)$, so it is the defining ideal $\mathcal{I}(L)\subset \mathcal{O}(V)$ of 
some closed subgroup $L$ of $V$. It is now clear from the construction that $R^J$ defines a perfect Hopf pairing between 
$\mathcal{O}(V)^J/\mathcal{I}$ and $\mathcal{O}(\bar{V})^{\bar{J}}$, so $L$ has codimension $1$ in $V$, $\iota(\mathbb{G}_a)\subseteq L$, and $\bar{L}=\pi(L)\subset \bar V$ is the support of $\bar{J}$. 
\end{proof}

For every integer $n\ge 0$ let $W(n)$ be the Weyl algebra of Gelfand-Kirillov dimension $2n$ (so $W(0)=\mathbb{C}$).

\begin{theorem}\label{noethdomunip}
Let $U$ be a unipotent algebraic group over $\mathbb{C}$. Let $J$ be a Hopf $2-$cocycle for $\mathcal{O}(U)$, and let $V$ be its support. The following hold:

1) $\mathcal{O}(U)_J\cong W(\dim(V))\ot \mathcal{O}(U/V)$, hence a Noetherian domain. 

2) $\mathcal{O}(V)_J\cong W(\dim(V))$ is a simple algebra.

3) $(\mathcal{O}(U)_J)^{op}\cong _{J^{-1}}\mathcal{O}(U)$.
\end{theorem}

\begin{proof}
It is sufficient to prove the first part of the theorem. The proof is by induction on the dimension $m$ of $U$, the cases $m=0,1$ being trivial. Let $1\to \mathbb{G}_a\xrightarrow{\iota}U\xrightarrow{\pi}\bar U\to 1$ be the central extension corresponding to the inclusion of Hopf algebras $A_{m-1}=\mathcal{O}(\bar U)\subset A_m=A$ (so $\bar U$ has dimension $m-1$), and let $\bar{L}\subseteq \bar U$ be the support of the restriction of $J$ to $\mathcal{O}(\bar U)$. By the induction assumption, 
$$\mathcal{O}(\bar U)_J\cong W(\dim(\bar{L}))\ot \mathcal{O}(\bar U/\bar{L})$$ as algebras, and by Lemma \ref{lemmaunip},
$\mathcal{O}(U)_J=\mathcal{O}(\bar U)_J[z_m;\delta_m]$. 

By a well known fact,
$$\Der(\mathcal{O}(\bar U)_J)=\left(\Der(W(\dim(\bar{L})))\ot \mathcal{O}(\bar U/\bar{L})\right)\bigoplus 
\Der(\mathcal{O}(\bar U/\bar{L}))$$
(as the center of a Weyl algebra is trivial).
Therefore, since every derivation of a Weyl algebra is inner, we have that 
$$\delta_m=\sum_j \ad w_j\ot y_j + \mu$$ for some $w_j\in W(\dim(\bar{L}))$, $y_j\in \mathcal{O}(\bar U/\bar{L})$, and $\mu\in \Der(\mathcal{O}(\bar U/\bar{L}))$. Explicitly, $$\delta_m(w\cdot z)=\sum_j (w_j\cdot w - w\cdot w_j)\ot y_j\cdot z + w\cdot \mu(z)$$
for every $w\in W(\dim(\bar{L}))$ and $z\in \mathcal{O}(\bar U/\bar{L})$.
But this implies that $$z_m\cdot (w\cdot z) - (w\cdot z)\cdot z_m=s_m\cdot (w\cdot z)- (w\cdot z)\cdot s_m+ w\cdot \mu(z),$$ where $s_m:=\sum_j w_j\cdot y_j\in \mathcal{O}(\bar U)_J$, and hence that 
$$
t_m\cdot (w\cdot z) - (w\cdot z)\cdot t_m=w\cdot \mu(z),
$$
where $t_m:=z_m - s_m$. For $z=1$, $t_m\cdot w - w\cdot t_m=0$, so we conclude that
\begin{equation}\label{AJ}
\mathcal{O}(U)_J\cong \mathcal{O}(\bar U)_J[z_m;\delta_m]\cong W(\dim(\bar{L}))\ot \mathcal{O}(\bar U/\bar{L})[z_m;\mu]. 
\end{equation}

Now set $\bar{V}:=\pi(V)$. There are two cases:

{\bf Case 1:}
If $\pi$ maps $V$ isomorphically onto $\bar{V}$ then $\bar{L}=\bar{V}\cong V$, so $\mathcal{O}(\bar U/\bar{L})$ is properly contained in $\mathcal{O}(U/V)$, which means by Theorem \ref{trivcen} that the center of $\mathcal{O}(\bar U)_J$ is properly contained in the center of $\mathcal{O}(U)_J$. We thus conclude that $z_m$ commutes with $\mathcal{O}(\bar U/\bar{L})$, so $$\mathcal{O}(\bar U/\bar{L})[z_m;\mu]\cong \mathcal{O}(U/V),$$ which implies by (\ref{AJ}) that $$\mathcal{O}(U)_J\cong W(\dim(\bar{L}))\ot \mathcal{O}(\bar U/\bar{L})[z_m;\mu]\cong W(\dim(V))\ot \mathcal{O}(U/V),$$ as desired.

{\bf Case 2:}
Suppose $\pi$ maps $V$ non-isomorphically onto $\bar{V}$. Then 
by Proposition \ref{support1ext}, $\mathcal{O}(\bar U/\bar V)\cong\mathcal{O}(U/V)$ is a proper subalgebra of $\mathcal{O}(\bar U/\bar{L})\cong \mathcal{O}(U/L)$, so by Theorem \ref{trivcen}, the center of $\mathcal{O}(U)_J$ is properly contained in the center of $\mathcal{O}(\bar U)_J$. Furthermore, $\mathcal{O}(\bar U/\bar{L})\cong \mathcal{O}(U/L)$ is mapped onto $\mathcal{O}(\bar V/\bar{L})\cong \mathcal{O}(V/L)$ under the {\em Hopf algebra surjection} $\mathcal{O}(U) \twoheadrightarrow \mathcal{O}(V)$.

By Proposition \ref{support1ext}, $\mathcal{O}(\bar V/\bar{L}) \cong \mathcal{O}(\mathbb{G}_a)$. Let $y\in \mathcal{O}(\bar U/\bar{L})$ be an element which is mapped to a nonzero primitive element in $\mathcal{O}(\bar V/\bar{L})$, and assume it has the minimal possible degree (with respect to the coradical filtration on $\mathcal{O}(U)$). Set $$Y:=\Delta(y)-y\ot 1-1\ot y,$$ and write $Y=\sum Y'\ot Y''$, where $\{Y'\}$ and $\{Y''\}$ are linearly independent sets. Note that $y$, and every $Y',Y''$ vanish at the point 
$e\in U$, i.e., they belong to $\mathcal{O}(U)^+$. Also, since $\mathcal{O}(\bar U/\bar{L})$ is a left coideal subalgebra in $\mathcal{O}(U)$, every $Y''$ belongs to $\mathcal{O}(\bar U/\bar{L})$.

Now, we have that
$$\mathcal{O}(\bar U/\bar{L})[z_m;\mu]\cong \left(\mathcal{O}(\bar U/\bar V)\ot \mathbb{C}[y]\right)[z_m;\mu],$$
where $z_m$ commutes with the elements of $\mathcal{O}(\bar U/\bar V)$ (but not with $y$), and by Lemma \ref{lemmaunip},
\begin{eqnarray*}
\lefteqn{z_m\cdot y - y\cdot z_m =
Q(z_m,y)1}\\ & + & \sum Z_m'Q(Z_m'',y) + \sum Y'Q(z_m,Y'') + \sum Z_m'Y'Q(Z_m'',Y'').
\end{eqnarray*} 
Since $y,Y''$ are central in $\mathcal{O}(\bar U)_J$, and $Z_m''$ is in $\mathcal{O}(\bar U)_J$, it follows that $Q(Z_m'',y)=0=Q(Z_m'',Y'')$ (see Remark \ref{Q}).
Moreover, since every $Y''$ is mapped to $\mathcal{O}(\bar V/\bar{L})^+ \cong \mathcal{O}(\mathbb{G}_a)^+$ under $\mathcal{O}(U) \twoheadrightarrow \mathcal{O}(V)$, it follows from the minimality of $y$ that every $Y''$ must be mapped to $0$. So $Q(z_m,Y'')=0$, too. We thus conclude that
$$z_m\cdot y-y\cdot z_m=Q(z_m,y)1\ne 0$$ (as $z_m$ and $y$ do not commute), and hence that 
$$\mathcal{O}(\bar U/\bar{L})[z_m;\mu]\cong \mathcal{O}(\bar U/\bar V)\ot \mathbb{C}[y][z_m;\mu]\cong \mathcal{O}(U/V)\ot W(1),$$ which implies by (\ref{AJ}) that $$\mathcal{O}(U)_J\cong W(\dim(\bar{L}))\ot \mathcal{O}(\bar U/\bar{L})[z_m;\mu]\cong W(\dim(V))\ot \mathcal{O}(U/V),$$ as desired.

This completes the proof of the theorem.
\end{proof}

Let us give some examples which illustrate the steps taken in the proof of Theorem \ref{noethdomunip}.

\begin{example}\label{addgr2} 
Let $U=\mathbb{G}_a\times \mathbb{G}_a$. Then $A=\mathcal{O}(U)=\mathbb{C}[x,y]$ is a polynomial Hopf algebra, i.e., both $x$ and $y$ are primitive elements. So in this case we have $$A_{0}=\mathbb{C}\subset A_1=\mathbb{C}[y]\subset A_2=A.$$

Let $\uu$ be the Lie algebra of $U$; it is abelian of dimension $2$, with basis $\{\partial /\partial x,\partial /\partial y\}$. Let $r:=\partial /\partial x \ot \partial /\partial y-\partial /\partial y \ot \partial /\partial x$, and let $\mathcal{O}(U)\ot \mathcal{O}(U)\xrightarrow{\eta}\mathcal{O}(U)\ot \mathcal{O}(U)$ denote the action by $e^{r/2}$ (it is well defined). Then as it is explained in \cite[Section 4]{EG1}, the composition
$$J:\mathcal{O}(U)\ot \mathcal{O}(U)\xrightarrow{\eta}\mathcal{O}(U)\ot \mathcal{O}(U)\xrightarrow{\varepsilon\ot \varepsilon}\mathbb{C}$$  
is a (minimal) Hopf $2-$cocycle for $\mathcal{O}(U)$. Explicitly, for every two polynomials $f,g\in \mathbb{C}[x,y]$
\begin{eqnarray*}
\lefteqn{J(f,g)=e^{\frac{1}{2}\left(\frac{\partial}{\partial x} \ot \frac{\partial}{\partial y}-\frac{\partial}{\partial y} \ot \frac{\partial}{\partial x}\right)}(f\ot g)_{|(0,0)}}\\
& = & f(0)g(0)+\frac{f_x(0)g_y(0)-f_y(0)g_x(0)}{2}+\cdots
\end{eqnarray*}
(since $\varepsilon$ is evaluation at the point $0\in \mathbb{C}^2$). In particular, it is straightforward to verify that $J(x,y)=1/2$ and $J(y,x)=-1/2$, and hence that $Q(x,y)=1$. Finally by Lemma \ref{lemmaunip}, $$(A_{0})_J=\mathbb{C}\subset (A_1)_J=\mathbb{C}[y]\subset (A_2)_J=A_J\cong \mathbb{C}[y][x;\delta],$$
and $x\cdot y-y\cdot x=Q(x,y)1=1$. We therefore conclude that $A_J\cong W(1)$ is the Weyl algebra.

More generally, let
$U=\mathbb{G}_a^{2n}$, $\uu=\mathbb{C}^{2n}$ with basis $p_1,\dots,p_n, q_1,\dots,q_n$, 
$r:=\sum_i p_i\wedge q_i\in \wedge ^2 \uu$, and $J:=e^{r/2}$. Then $\mathcal{O}(U)_J$ is the Weyl algebra $W(n)$ ({\em the
Moyal-Weyl quantization}).
\end{example}

\begin{example}\label{heisen} 
Let $U$ be the $3-$dimensional Heisenberg group. Then $A=\mathcal{O}(U)=\mathbb{C}[x,y,z]$ is a polynomial algebra, where both $x$ and $y$ are primitive elements, and $\Delta(z)=z\ot 1+1\ot z+x\ot y$. So in this case we have $$A_{0}=\mathbb{C}\subset A_1=\mathbb{C}[y]\subset A_2=\mathbb{C}[y,x]\subset A_3=A.$$

Consider the surjective homomorphism of Hopf algebras $$\mathcal{O}(U)\xrightarrow{\pi} \mathbb{C}[Z,Y]=\mathcal{O}(\mathbb{G}_a\times \mathbb{G}_a),\,\,
x\mapsto 0,\,y\mapsto Y,\,z\mapsto Z,$$ and let $J$ be the minimal Hopf $2-$cocycle for $\mathbb{C}[Z,Y]$ from Example \ref{addgr2}. Viewing $J$ as (a non-minimal) Hopf $2-$cocycle for $A$ via $\pi$, we see that its support $V\cong \mathbb{G}_a\times \mathbb{G}_a$ is the closed subgroup of $U$ corresponding to $\pi$. Also, $J(z,y)=J(Z,Y)=1/2$, $J(y,z)=J(Y,Z)=-1/2$, $J(x,y)=J(0,Y)=\varepsilon(Y)=0$, and $J(x,z)=J(0,Z)=\varepsilon(Z)=0$. So $Q(z,y)=1$ and $Q(y,x)=Q(z,x)=0$. 

Now by Lemma \ref{lemmaunip}, $$\mathbb{C}\subset (A_1)_{J}=\mathbb{C}[y]\subset (A_2)_{J}=\mathbb{C}[y][x;\delta_1]\subset A_{J}\cong \mathbb{C}[y][x;\delta_1][z;\delta_2],$$
$x\cdot y-y\cdot x=Q(x,y)1=0$, $z\cdot y-y\cdot z=Q(z,y)1+xQ(y,y)=1$, and $z\cdot x-x\cdot z=Q(z,x)1+xQ(y,x)=0$. We therefore conclude that $(A_2)_{J}=\mathbb{C}[y,x]$ and 
$$A_{J}\cong \mathcal{O}(V)_J\ot \mathcal{O}(U/V)\cong W(1)\ot \mathbb{C}[x].$$
Note also that the restriction of $J$ to $A_2$ is trivial, so the dimension of its support is $\dim(V)-2$ ($=0$).
\end{example}

\begin{example}
Let $U$ be a unipotent algebraic group over $\mathbb{C}$, and let $\uu$ be its Lie algebra. Let $J(r)$ be the Hopf 2-cocycle for $\mathcal{O}(U)$ corresponding to the solution $r\in \wedge^2\uu$ of the CYBE \cite{EG2} (see also Section 5 below), let $V$ be its support, and let $\mathfrak{v}$ be the Lie algebra of $V$. We have that $\mathcal{O}(U)_{J(r)}\cong W(\mathfrak{v})\ot \mathcal{O}(U/V)$. In particular, the center of $\mathcal{O}(U)_{J(r)}$ is equal to $\mathcal{O}(U/V)$, and if $r\in \wedge^2\uu$ is nondegenerate then the algebra $\mathcal{O}(U)_{J(r)}$ is isomorphic to a Weyl algebra. Indeed, this follows from Theorem \ref{noethdomunip}.
\end{example}

\subsection{The general case} Let $G=T\times U$ be a connected nilpotent algebraic group over $\mathbb{C}$, where $T$ is a torus and $U$ is a unipotent group. Recall that $\mathcal{O}(G)=\mathcal{O}(T)\ot \mathcal{O}(U)$ as Hopf algebras. Let $\mathcal{O}(U)=\mathbb{C}[z_1,\dots,z_m]$, and $\mathcal{O}(T)=\mathbb{C}[X(T)]=\mathbb{C}[x_1^{\pm 1},\dots,x_k^{\pm 1}]$. We shall assume that $1=\deg (z_1)=\deg (z_2)\le \deg (z_3)\le \cdots \le \deg (z_m)$ (with respect to the coradical filtration on $\mathcal{O}(U)$). For every $1\le j\le m$, write $\Delta(z_j)=z_j\ot 1 +1\ot z_j +\sum Z_j'\ot Z_j''$.

\begin{lemma}\label{lemmanilp}
Let $G=T\times U$ be a connected nilpotent algebraic group over $\mathbb{C}$, and let $J$ be a Hopf $2-$cocycle for $\mathcal{O}(G)$. The following hold:

1) $\mathcal{O}(T)_J$ and $\mathcal{O}(U)_J$ are subalgebras of $\mathcal{O}(G)_J$.

2) For every $1\le i\le k$ and $1\le j\le m$, 
\begin{equation}\label{gp}
x_i\cdot z_j\cdot x_i^{-1} = z_j + p_{ij}, 
\end{equation}
where $p_{ij}:=Q(x_i,z_j)1+ \sum Z_j'Q(x_i,Z_j'')\in \mathcal{O}(U)$ has smaller degree than $\deg (z_j)$.
\end{lemma}

\begin{proof}
Follow from (\ref{multj}) in a straightforward manner.
\end{proof}

Set $p_{i}:=(p_{i1},\dots, p_{im})$ for every $1\le i\le k$, and ${\bf x}^{\alpha}:=x_1^{\alpha_1}\cdots x_k^{\alpha_m}$ for every $\alpha=(\alpha_1,\dots,\alpha_k)\in \mathbb{Z}^k$.

\begin{theorem}\label{noethdomnilp}
Let $G=T\times U$ be a connected nilpotent algebraic group over $\mathbb{C}$. Let $J$ be a Hopf $2-$cocycle for $\mathcal{O}(G)$, let $H$ be its support, and let $V\subseteq U$ be the support of the restriction of $J$ to $U$. Let $n:=\dim(V)$, and set $\lambda_{ij}:=R^J(x_i,x_j)$, $1\le i,j\le k$. The following hold:

1) $\mathcal{O}(G)_J\cong (\mathcal{O}(U/V)\ot W(n)) \#_J \mathbb{C}[X(T)]$ is a crossed product with $2-$cocycle $J$, where $X(T)$ acts on $W(n)\ot \mathcal{O}(U/V)$ via automorphisms as in (\ref{gp}).

2) $\mathcal{O}(G)_J$ is a Noetherian domain.
 
3) $\mathcal{O}(H)_J$ is a simple algebra.

4) $(\mathcal{O}(G)_J)^{op}\cong _{J^{-1}}\mathcal{O}(G)$.
\end{theorem}

\begin{proof}
1) Follows from Lemma \ref{lemmanilp}, Theorem \ref{noethdomunip}, and Theorem \ref{noethdomtori}.

2) Follows from 1), and \cite[Theorem 5.12]{MR}.

3) Let $\mathcal{O}(U/V)=\mathbb{C}[z_1,\dots, z_l]$, $W(n)=\mathbb{C}[u_1,v_1,\dots,u_n,v_n]$, and set ${\bf z}:=(z_1,\dots, z_l,u_1,v_1,\dots,u_n,v_n)$. By fixing an order on the monomials in $u_1,v_1,\dots,u_n,v_n$, we may view elements in $W(n)$ as polynomials.

Let $I$ be a nonzero ideal of $A:=\mathcal{O}(H)_J$, and 
let $a=\sum_{\alpha} f_{\alpha}({\bf z}) {\bf x}^{\alpha}$ ($0\ne f_{\alpha}({\bf z})\in \mathcal{O}(U/V)\ot W(n)$) be a nonzero element in $I$ with minimal length (= number of distinct $\alpha$ involved in the expression). We may assume that $f_0({\bf z})\ne 0$, and choose $a$ to have a nonzero free coefficient $f_0({\bf z})$ of minimal degree. 

Consider the element $z_j\cdot a - a\cdot z_j\in I$ for every $1\le j\le l$. Since $z_j\cdot f_0({\bf z})=f_0({\bf z})\cdot z_j$, we see that $z_j\cdot a - a\cdot z_j$ has shorter length than $a$. Thus, $a$ commutes with $z_j$ for every $1\le j\le l$.

Consider the element $u_j\cdot a - a\cdot u_j\in I$ for every $1\le j\le n$. Since $$u_j\cdot f_0({\bf z})=f_0({\bf z})\cdot u_j+\frac{\partial(f_0({\bf z}))}{\partial v_j},$$ we see that $u_j\cdot a - a\cdot u_j$ has free coefficient with smaller degree than $f_0({\bf z})$. Thus, $a$ commutes with $u_j$ for every $1\le j\le n$.

Similarly, $a$ commutes with $v_j$ for every $1\le j\le n$.

Finally consider the element $x_i\cdot a\cdot x_i^{-1}-a\in I$ for every $1\le i\le k$. By Lemma \ref{lemmanilp},
$$x_i\cdot a\cdot x_i^{-1}=\sum_{\alpha} (\Pi_{j=1}^k \lambda_{ij}^{\alpha_j})f_{\alpha}({\bf z}+p_{i}){\bf x}^{\alpha},$$ so
$$x_i\cdot a\cdot x_i^{-1}-a=\sum_{\alpha} \left((\Pi_{j=1}^k \lambda_{ij}^{\alpha_j})f_{\alpha}({\bf z}+p_{i})-f_{\alpha}({\bf z})\right) {\bf x}^{\alpha}.$$ If $a$ commutes with $x_i$ for every $1\le i\le k$ then $a$ is central in $A$, hence is a nonzero constant by Theorem \ref{trivcen}, and we are done. Otherwise, let $i$ be such that $a$ does not commute with $x_i$. Then 
$0\neq x_i\cdot a\cdot x_i^{-1}-a\in I$ must have the same length as $a$. But $f_{0}({\bf z}+p_i)-f_{0}({\bf z})$ has smaller degree than $f_{0}({\bf z})$, which is a contradiction.
\end{proof}

\begin{example} 
Let $G=\mathbb{G}_m\times \mathbb{G}_m\times \mathbb{G}_a$. Then $\mathcal{O}(G)=\mathbb{C}[x^{\pm 1},y^{\pm 1},z]$, where $x,y$ are grouplike elements and $z$ is a primitive element. 
The Lie algebra of $G$ is abelian and has basis $X:=x\frac{\partial}{\partial x}$, $Y:=y\frac{\partial}{\partial y}$, $Z:=\frac{\partial}{\partial z}$. Consider the classical $r-$matrix $r:=X\wedge (hY+Z)$, where $h\in \mathbb{C}^*$, and let $J:=e^{r/2}$. Then $\mathcal{O}(G)_J$ has generators $x,y,z$ with relations $xy=e^hyx$, $yz=zy$, and $xz=(z+1)x$. Setting $D:=x^{-1}z$, we have $Dy=e^{-h}yD$, $xy=e^hyx$, and $[D,x]=1$, so $\mathcal{O}(G)_J$ is a smash product of $\mathbb{C}[\mathbb{Z}]$ ($\mathbb{Z}$ is generated by $y$) with 
the algebra of differential operators on $\mathbb{G}_m$ generated by $D$ and $x^{\pm 1}$.
 
Note that if $h\notin \pi i \mathbb{Q}$ then 
$J$ is minimal, 
and $\mathcal{O}(G)_J$ is a simple algebra. 
\end{example}

\section{Hopf $2-$cocycles for nilpotent algebraic groups}\label{nilpsec}

Recall that using Etingof--Kazhdan quantization theory \cite{EK1,EK2,EK3}, it was proved in \cite{EG2} that gauge equivalence classes of Hopf 2-cocycles for unipotent algebraic groups $U$ over $\mathbb{C}$ are in one to one correspondence with solutions $r\in\wedge^2\Lie(U)$ to the CYBE (\ref{cybe}). We can now give an alternative proof to this result, which uses Theorem \ref{noethdomunip}, together with \cite[Theorem 5.5]{EG1} and a result of Drinfeld (see Section \ref{drinres}), instead.

\begin{theorem}\label{qfunip}
Let $U$ be a unipotent algebraic group over $\mathbb{C}$, and let $J$ be a minimal Hopf $2-$cocycle for $\mathcal{O}(U)$. Then $U$
admits a left invariant symplectic structure $\omega$, i.e., $\uu:=\Lie(U)$ is a quasi-Frobenius Lie algebra with $\omega\in\wedge^2\uu^*$. In particular $\dim(U)$ is even.
\end{theorem}

\begin{proof} \footnote{I am indebted to Pavel Etingof and Eric Rains for their help with this proof.}
Consider the left action of $U$ on $\mathcal{O}(U)_J$ by algebra automorphisms (via left translations). It yields an action of $\uu$ on $\mathcal{O}(U)_J$ by derivations. Since by Theorem \ref{noethdomunip}, $\mathcal{O}(U)_J$ is a Weyl algebra, this action determines a $2-$cocycle $\omega\in \wedge^2\uu^*$. Namely, for every $x\in\uu$ choose $w_x\in \mathcal{O}(U)_J$ such that $x$ acts as the inner derivation given by $w_x$. Then $w_{[x,y]}=\omega(x,y)+ w_x w_y- w_y w_x$ for every $x,y\in \uu$. Now it is well known that the $2-$cocycle $\omega$ defines a 
central extension $\hat \uu$ of $\uu$ by the $1-$dimensional Lie algebra $\mathbb{C}$ (see Section \ref{rext}). Let $z\in \hat \uu$ be the corresponding central element. 

The algebra $U(\hat \uu)/(z-1)$ has a conjugation action of $U$, and we see that the above assignment $x\mapsto w_x$ determines an isomorphism $U(\hat \uu)/(z-1)\xrightarrow{\cong} \mathcal{O}(U)_J$ of $U-$algebras. In particular, $U(\hat \uu)/(z-1)$ is isomorphic to the regular representation of $U$ as a module. 
By using the symmetrization map $\Sym(\hat \uu)\xrightarrow{\cong} U(\hat \uu)$, which is a $U-$module map, we see that $\Sym(\hat \uu)/(z-1)$ is isomorphic to the regular representation of $U$. This implies that the Poisson center $(\Sym(\hat \uu)/(z-1))^U$ is $1-$dimensional. So the only $U-$invariant function on the set $L$ 
of elements $f\in \hat \uu^*$ satisfying $f(z)=1$ is a constant.
Take an orbit $X$
of $U$ on $L$. It is closed by the theorem of Kostant and Rosenlicht. Let $a_1,\dots,a_m$ be generators of the ideal of $X$, and let $V$ be the $U-$module generated by them. Since 
the action of $U$ is locally finite, $V$ is finite dimensional. Let $f\in V$ be a nonzero $U-$invariant element (which exists because $U$ is unipotent). Then $f$ is a nonzero constant (since all invariants are constant). So $X=L$ is a coadjoint orbit of $V$, on which $U$ acts freely. So $U$ has a left invariant symplectic structure (the canonical $U-$invariant symplectic structure of $L$), which is given at the identity by $\omega$. So $\omega$ is nondegenerate and $\uu$ is quasi-Frobenius, as claimed. 
\end{proof}

\begin{corollary}\label{revis}
Let $U$ be a unipotent algebraic group over $\mathbb{C}$ with Lie algebra $\uu$. The following sets are in bijection one with each other:

1) Gauge equivalence classes of Hopf $2-$cocycles for $\mathcal{O}(U)$.

2) Equivalence classes of pairs $(V,\omega)$, where $V$ is a closed subgroup of $U$ and $\omega$ is a left invariant symplectic form on $V$. 

3) Equivalence classes of solutions $r\in \wedge^2\uu$ to the CYBE (\ref{cybe}).
\qed
\end{corollary}

From now on let $G=T\times U$ be a connected nilpotent algebraic group over $\mathbb{C}$, where $T$ is a torus and $U$ is a unipotent group. Let $\g$, $\mathfrak{t}$ and $\uu$ be the Lie algebras of $G$, $T$ and $U$, respectively. Then $\mathfrak{t}$ is abelian, $\uu$ is nilpotent, and $\g=\mathfrak{t}+\uu$ is a direct sum of Lie algebras.

\begin{theorem}\label{nilp} Gauge
equivalence classes of Hopf 2-cocycles for $\mathcal{O}(G)$ are in correspondence with solutions of the
CYBE (\ref{cybe}) in $\wedge^2\g$, given by $J_f(r)\leftarrow r$, where $f$ is any universal quantization formula.
\end{theorem}

\begin{proof}
Let $J$ be a Hopf $2-$cocycle for $\mathcal{O}(G)$, and let $J_s$ be its restriction to $\mathcal{O}(T)$. By \cite{EG1}, $J_s$ corresponds to a solution of the CYBE in $\wedge^2\mathfrak{t}$ (not uniquely though; see, e.g., Example \ref{multgr2} above).
Viewing $J_s$ as a Hopf $2-$cocycle for $\mathcal{O}(G)$ we can consider the Hopf $2-$cocycle $J_s^{-1}*J$ for $\mathcal{O}(G)$. It is clear that the Hopf $2-$cocycle $J_s^{-1}*J$ is trivial on $\mathcal{O}(T)$, and hence defines a unipotent fiber functor on $\Rep(G)$ (i.e., a fiber functor which coincides with the standard one on the semisimple subcategory $\Rep(G/U)=\Rep(T)$). 
Therefore by \cite[Theorem 3.2]{EG2}, $J_s^{-1}*J$ corresponds to a (unique) solution of the CYBE (\ref{cybe}) in $\g\wedge \mathfrak{u}$, and hence $J=J_s*(J_s^{-1}*J)$ corresponds to a solution of the CYBE (\ref{cybe}) in $\wedge^2\g$, as claimed.
\end{proof}

It thus remains to classify solutions to the CYBE (\ref{cybe}) in $\wedge^2\g$.

\begin{proposition}\label{solcybe}
Every solution to the CYBE (\ref{cybe}) in $\wedge^2\g=\wedge^2(\mathfrak{t}+\uu)$ has the form $s+(w-w_{21})+r$, where $s\in \wedge^2\mathfrak{t}$ is arbitrary, 
$r$ is a solution to the CYBE (\ref{cybe}) in $\wedge^2 \uu$, and $w=\sum t_i\otimes u_i\in \mathfrak{t}\otimes \uu$, where $\{t_i\}$ is a basis of $\mathfrak{t}$, and $u_i$ are some elements of $\uu$ such that $[u_i\otimes 1+1\otimes u_i,r]=0$ for every $i$.
\end{proposition}

\begin{proof}
Suppose ${\bf r}\in\wedge^2\g$ is a solution to the CYBE (\ref{cybe}). Fix a basis $\{t_i\}$ of $\mathfrak{t}$, and write ${\bf r}=s+\sum (t_i\otimes u_i -u_i\ot t_i)+r$, where $s\in \wedge^2\mathfrak{t}$, $u_i\in \uu$, and 
$r\in \wedge^2\mathfrak{u}$. Using that $\mathfrak{t}$ is abelian, and $[\mathfrak{t},\uu]=0$, it is straightforward to verify that
\begin{eqnarray*}
\lefteqn{[{\bf r}_{12},{\bf r}_{13}]+[{\bf r}_{12},{\bf r}_{23}]+[{\bf r}_{13},{\bf r}_{23}]}\\
& = &  [r_{12},r_{13}]+\sum [r_1,u_i]\ot t_i\ot r_2 + \sum [u_i,r_1]\ot r_2\ot t_i \\
& + & [r_{12},r_{23}]+\sum r_1\ot [u_i,r_2]\otimes t_i+ \sum t_i\otimes [u_i,r_1]\ot r_2\\
& + &  [r_{13},r_{23}]+\sum t_i\ot r_1\otimes [u_i,r_2]+\sum r_1\ot t_i\otimes [r_2,u_i]\\
& = &  [r_{12},r_{13}]+[r_{12},r_{23}]+[r_{13},r_{23}]\\
& + &  \sum t_i\otimes [u_i\otimes 1+1\otimes u_i,r]+\sum [u_i\otimes 1+1\otimes u_i,r]\otimes t_i\\
& + & \sum \left([r_1,u_i]\ot t_i\ot r_2 + r_1\ot t_i\otimes [r_2,u_i]\right),
\end{eqnarray*}
from which the claim readily follows.
\end{proof}

\begin{remark}\label{rmkbij}
Observe that since a connected nilpotent algebraic group $G$ over $\mathbb{C}$ has only $1-$dimensional irreducible representations, all fiber functors on $\Rep(G)=\Corep(\mathcal{O}(G))$ are classical, so all of them are obtained by twisting using Hopf $2-$cocycles for $\mathcal{O}(G)$.
\end{remark}

\section{Concluding questions}

We conclude with a number of ring-theoretical questions concerning the structure of the algebras $\mathcal{O}(G)_J$, which are motivated by the results of the previous sections.

\begin{question}\label{noeth}
Let $G$ be an affine algebraic group over $\mathbb{C}$, and let $J$ be a Hopf $2-$cocycle for $\mathcal{O}(G)$ with support $H$.

1) Is the algebra $\mathcal{O}(H)_J$ simple? What are the simple algebras which can be realized in this way? 


2) Is the algebra $\mathcal{O}(G)_J$ Noetherian? More generally,  suppose $G$ acts on a Noetherian algebra $A$ with multiplication map $m$ (e.g., $A=\mathcal{O}(X)$, where $X$ is an affine Noetherian scheme such as a scheme of finite type or an algebraic variety), and let $A_J$ be the algebra with multiplication $m\circ J$ (here $J$ is viewed as an operator on $A\otimes A$ defined by the Hopf $2-$cocycle $J$). Is $A_J$ Noetherian? 

3) If $G$ is connected, when is $\mathcal{O}(G)_J$ a domain?

4) Does $\mathcal{O}(G)_J$ have zero Jacobson radical, or at
least zero prime radical?

5) Does $G$ act transitively on the set
of primitive/maximal ideals of $\mathcal{O}(G)_J$? Are primitive ideals of $\mathcal{O}(G)_J$ maximal? (This would follow from $G$ acting transitively on the primitive ideals since any maximal ideal is primitive.)
\end{question}

\begin{remark}\label{concrmk}
Assume $J$ is a Hopf $2-$cocycle of {\em finite rank}. 

1) By \cite{M}, $\mathcal{O}(H)_J\cong Mat_{|H|^{1/2}}(\mathbb{C})$ is simple.

2) The algebra $\mathcal{O}(G)_J$ is a finite module over its center $\mathcal{O}(G/H)$, so the answer to Question 2) is yes.  

3) Question \ref{noeth} 3) is equivalent to the following question: Let $G$ be a connected group, $H$ a finite subgroup of central type, and $J$ a minimal Hopf $2-$cocycle for $H$. Is it true that $\mathbb{C}(G)_J$ (= the field of rational functions on $G$ twisted by $J$) is a division algebra? For example, if $G=\mathbb{G}_m\times \mathbb{G}_m$ and $H=\mathbb{Z}/d\mathbb{Z}\times \mathbb{Z}/d\mathbb{Z}$ then this is true (see Section 4.1). However here is an example where $\mathcal{O}(G)_J$ is not a domain for connected $G$.

Let $G:=\PGL_2(\mathbb{C})$ be the group of automorphisms of the projective line $\mathbb{C}P^1$. Take the subgroup $H:=\mathbb{Z}/2\mathbb{Z}\times \mathbb{Z}/2\mathbb{Z}$ in $G$ generated by the maps $z\mapsto -z$ and $z\mapsto 1/z$ of $\mathbb{C}P^1$ to itself. Let $J$ be the standard twist supported on $H$, and $A:=\mathbb{C}(G)_J$. We claim that $A$ is not a division algebra,
namely it contains zero divisors.

Indeed, let $B$ be the upper triangular subgroup of $G$ (the stabilizer of $\infty$ in $\mathbb{C}P^1$). It suffices to check that
the subalgebra $\mathbb{C}(B\backslash G)_J=\mathbb{C}(z)_J$ has zero divisors. But this is a central simple algebra with center a function field of a (rational) curve $\mathbb{C}(u)$, where $u:=z^2+z^{-2}$. It is well known that there is no nontrivial central division algebras over $\mathbb{C}(u)$ (its Brauer group is trivial).
So $\mathbb{C}(z)_J=Mat_2(\mathbb{C}(u))$, hence contains zero divisors, as claimed.

4) Since $H$ is of central type, $\mathcal{O}(G)_J$ is an Azumaya algebra over $\mathcal{O}(G/H)$, so the answer to Question 4) is yes.

5) The answer to Question 5) is yes for the same reason as in 4).
\end{remark}




\end{document}